\documentclass[10pt]{iopart}
\usepackage[top=3cm,bottom=2cm,left=2.5cm,right=3cm]{geometry}
\usepackage{ntheorem}
\usepackage{amsfonts}
\usepackage{amssymb}
\usepackage{enumerate}
\usepackage{ragged2e}
\usepackage{caption}
\usepackage{color}
\usepackage{hyperref}
\hypersetup{
    linkcolor=blue,
    filecolor=gray,
    urlcolor=blue,
    citecolor=blue,
}

\usepackage{amssymb}
  \theoremsymbol{\mbox{$\blacksquare$}}
\usepackage{euscript,eufrak,verbatim}
\usepackage[all,pdf]{xy}
\usepackage{graphicx}
\usepackage{float}
\usepackage[final]{epsfig}
\usepackage{url}

\usepackage{setspace}
\newtheorem{theorem}{Theorem}

\newtheorem{lemma}{Lemma}
\newtheorem{proposition}{Proposition}

\newtheorem*{proof}{Proof}{\upshape}

\newtheorem{remark}{Remark}

\newenvironment{myproof}
{
\upshape
}

\renewcommand{\theequation}{\thesection .\ \arabic{equation}}

\begin{document}
\title[Thick hyperbolic repelling invariant Cantor set and wild attractor]{Thick hyperbolic repelling invariant Cantor set and wild attractor}

\author{ Yiming Ding, Jianrong Xiao}

\address{Center for Mathematical Sciences and Department of Mathematics, Wuhan University of Technology, Wuhan 430070, China}

\ead{dingym@whut.edu.cn and xiaojr@whut.edu.cn}
\vspace{10pt}
\begin{indented}
\item[]
\end{indented}

\begin{abstract}
Let $D$ be the set of $\beta \in (1, 2]$ such that $f_\beta$ is a symmetric tent map with finite critical orbit. For $\beta \in D$, by operating Denjoy like surgery on $f_{\beta}$, we constructed a $C^1$ unimodal map $\tilde{g}_\beta$ admitting a thick hyperbolic repelling invariant Cantor set which contains a wild Cantor attractor. The smoothness of $\tilde{g}_\beta$ is ensured by the effective estimation of the preimages of the critical point as well as the prescribed lengths of the inserted intervals. Furthermore, $D$ is dense in $(1, 2]$, and $\tilde{g}_\beta$  can not be $C^{1+\alpha}$ because the hyperbolic repelling invariant Cantor set of $C^{1+\alpha}$ map has Lebesgue measure equal to zero.
\end{abstract}

%
%
%
%
%
\section{Introduction}
\renewcommand{\theproposition}{\Alph{proposition}}
\renewcommand{\thetheorem}{\Alph{theorem}}
\renewcommand{\thelemma}{\Alph{lemma}}
\noindent

Let $f:I\rightarrow I$ be a $C^r, r\geq 1$ map, where $I$ is a compact interval of the real line. A subset $K\subset I$ is a hyperbolic repelling(or for short, hyperbolic) set of $f$ if $K$ is forward invariant and there exist constants $C>0$ and $\lambda>1$ such that $|Df^n(x)|>C\lambda^n$ for all $x\in K$ and all $n\in N$. We call $K$ a thick set if it admits positive Lebesgue measure. In  \cite{xia2006hyperbolic}, Xia proved that if a $C^{1+\alpha}, \alpha>0$, volume preserving diffeomorphism on a compact connected manifold has a hyperbolic invariant set with positive volume, then the map is Anosov. This is not necessarily true for $C^1$ map. In  \cite{morales2007examples}, Morales constructed examples of singular hyperbolic attracting sets. The construction involves Dehn surgery, templates and geometric models. Whether there exists $C^1$ unimodal map admitting thick hyperbolic repelling invariant Cantor set on interval of the real line is an interesting problem.

Attractor is one of the main concepts in the theory of dynamical systems. According to Milnor  \cite{milnor1985concept}, a forward invariant closed set $A$ of $f$ is called a metric attractor provided the basin of attraction $\mathbb{B}(A,f):=\{x: \omega_f(x)\subset A\}$ ($\mathbb{B}(A)$ for short) of $A$ has positive Lebesgue measure and there exists no forward invariant closed set $A_1$ strictly contained in $A$ has this property. In particular, if a forward invariant closed set has positive Lebesgue measure, then it may contain metric attractor. A forward invariant closed set $A$ is called a topological attractor provided the basin of attraction $\mathbb{B}(A)$ of $A$ is a residual set (i.e., the intersection of countable open dense sets), while the basin $\mathbb{B}(A_0)$ is meager (i.e., the union of countable closed nowhere dense sets) for every forward invariant closed set $A_0$ strictly contained in $A$. The set $A$ is called wild attractor if $A$ is a metric attractor but not a topological attractor.

A topological attractor must be a metric attractor, but the converse is not necessarily true. In  \cite{milnor1985concept}, Milnor asked whether there exists a wild attractor. Blokh and Lyubich \cite{blokh1991measurable} studied metric attractors under the assumption that $f$ has negative Schwarzian derivative. Their work was extended to maps in the class $\mathbb{A}$ after \cite{kozlovski2000getting}, where $\mathbb{A}$ denotes the collection of unimodal maps that $f$ are $C^3$ outsize the critical points $c$. In particular, it was shown that if the metric attractor is a Cantor set $X$, then $X$ must coincide with $\omega(c)$ and the topological entropy $h_{top}(f|_X)=0$. In \cite{bruin1996wild}, Van Strien showed $X=\omega(c)$ is a wild Cantor attractor with zero Lebesgue measure for the unimodal interval $C^2$ map with a critical point $c$ and with the Fibonacci combinatorics as the critical order $l$ sufficiently large. This result was generalized to unimodal maps with "Fibanacci-like" combinatorics by Bruin  \cite{bruin1998topological}. In  \cite{bruin1997adding}, Bruin, Keller and Pierre constructed a unimodal map $f$ together with a symbolic dynamical system $(\Omega, T)$ such that $f$ has a wild attractor $\omega(c)$ and $f|\omega(c)$ is a factor of $(\Omega, T)$. In \cite{bruin1998topological}, Bruin showed that smooth unimodal maps for which $k-Q(k)$ is unbounded cannot have any wild attractors for any large but finite value of the critical order, then he exhibited an example of a wild attractor in \cite{bruin2015wild}. In \cite{li2014new}, Li and Wang introduced the class $W_m$ which is different from Fibonacci type as $m\geq 2$, and left a question that whether maps in $W_m, m\geq 2$ have a wild attractor. In \cite{ding2022alpha}, Ding and Sun investigated the topological behaviors of continuous maps with one topological attractor.

In \cite{murdock2005map}, a map $f$ with invariant Cantor set $\Lambda$ admitting positive Lebesgue measure was constructed, where $f$ is differentiable everywhere except on a countable set $B$ and $|f'(x)|=2$ for all $x$ in $\Lambda\backslash B$. We asked that whether there exists a $C^1$ map $f$ with invariant Cantor set admitting positive Lebesgue measure, which contains a wild attractor.

We are interested in a family of symmetric tent maps $f_{\beta}(x):I\rightarrow I,\ I=[0,1]$ with critical (turning) point $c=\frac{1}{2}$ and $1<\beta\leq 2$ of the form
$$
f_{\beta}(x)=\left \{
\begin{array}{ll}
\beta x,& 0\leq x \leq \frac{1}{2}\\
\beta(1-x),& \frac{1}{2}< x \leq 1.
\end{array}
\right.
$$

In the case $1<\beta\leq 2$, $f_\beta$ has a core $P_\beta:=[\beta(1-\frac{\beta}{2}),\frac{\beta}{2}]$. Notice that $f_\beta(P_\beta)=P_{\beta}$. Recall that a unimodal map $f$ is renormalizable provided there exists a proper restrictive interval $J$ and $n>1$ such that $f^n(J)\subset J$ and $f^n\mid J$ is again a unimodal map (we relax the definition of unimodal to allow for a map to be decreasing to the left of the critical point and increasing to the right), where $J$ is called renormalization interval. It is obviously seen that any renormalization interval admits the form $[a,\hat{a}]$ such that $f(a)=f(\hat{a})$. We always denote the closed interval by $[a,\hat{a}]$ with endpoints $a$ and $\hat{a}$, regardless of $a>\hat{a}$ or $a<\hat{a}$. According to  \cite{brucks2004topics}, $f_\beta$ can be renormalized finite times via period-doubling renormalization as $\beta\in (1,\sqrt{2}]$. More precisely, a tent map $f_\beta$ with $\beta\in(\sqrt{2},2]$ is not renormalizable. If $\sqrt{2}< \beta^m \leq 2$ for some $m:=2^k, k\geq 1$, then $f_\beta$ is $k$ times renormalizable, and there exists an restrictive interval $R_\beta\subset P_\beta$ such that $f_\beta^{2^k}(R_\beta)\subset R_\beta$. Gao and Gao\cite{gao2021fair} investigated the tent map and characterized the regularity of fair entropy precisely.

In this paper, we will construct a family of $C^1$ unimodal map admitting thick hyperbolic repelling invariant Cantor set which contains a wild attractor by operating Denjoy like surgery on symmetric tent map. More precisely, let $D$ be the set of $\beta \in (1, 2]$ such that $f_\beta$ is a symmetric tent map with finite critical orbit. For $\beta \in D$, by inserting suitable intervals at the preimages of the critical periodic orbit of $f_{\beta}$, and defining suitable maps on them, we construct a $C^1$ unimodal map $g_\beta$ admitting thick hyperbolic repelling invariant Cantor set which contains a wild Cantor attractor of $\tilde{g}_\beta$.  The smoothness of $\tilde{g}_\beta$ is ensured by the effective estimation of the preimages of the critical point as well as the prescribed lengths of the inserted intervals. Furthermore, $D$ is dense in $(1, 2]$, and $\tilde{g}_\beta$  can not be $C^{1+\alpha}$ because the hyperbolic repelling invariant Cantor set of $C^{1+\alpha}$ map has Lebesgue measure equal to zero. In fact, the map $\tilde{g}_\beta$ constructed above has two metric attractors. One of them is a attracting periodic orbit which  absorbs the interiors of all of the inserted intervals except for a nowhere dense countable set $\tilde{\Lambda}_{\beta}$, it is also a topological attractor.  The collection of points do not attracted by the attracting periodic orbit is $ \tilde{A}_{\beta} \bigcup \tilde{\Lambda}_{\beta}$, where $ \tilde{A}_{\beta}$ is a hyperbolic repelling invariant Cantor set with positive Lebesgue measure, and $\tilde{A}_{\beta}$ contains an metric attractor $\tilde{A}_{\beta}^\ast$. $\tilde{A}_{\beta}^\ast$ is wild because its basin is nowhere dense.

For any Borel set $A$ of the real line, $m(A)$ denotes the Lebesgue measure of $A$ and $\sharp(A)$ denotes the cardinal of $A$. If $A$ is an interval, we denote the length of $A$ by $|A|$.

Let
$$\lambda_{\beta,1}=\frac{1}{\beta}, \quad \lambda_{\beta,n}=\frac{n-1}{\beta(n+1)}, n>1, \quad a_{\beta,n}=\prod\limits_{k=1}^n\lambda_{\beta,k}=\frac{1}{\beta^n}\frac{2}{n(n+1)}.$$
Notice that
$$a_{\beta,n}\rightarrow 0,\quad \frac{a_{\beta,n-1}}{a_{\beta,n}}=\frac{\beta(n+1)}{n-1}>\beta, \quad \frac{a_{\beta,n-1}}{a_{\beta,n}}\rightarrow \beta \quad (n\rightarrow +\infty).$$

Suppose the orbit of the critical point $c$ of $f_\beta$ is finite. Let the orbit of $c$ be $orb(c)=\{c, c_1, \cdots , c_t, \cdots , c_{t+m-1}\}$, where $c_i=f^i_\beta(c)$ and $c_t$ is $m$-periodic. Put
$$C_{\beta,n}=\{x|f_{\beta}^n(x)=c_t,\ f_{\beta}^s(x)\neq c_t, 0\leq s<n, \  x\in I\}, n\geq 0,$$
$$C_{\beta}=\cup_{n=0}^{\infty}C_{\beta,n},\ \ F_{\beta,n}=\sharp\{C_{\beta,n}\}.$$
Notice that $C_{\beta,n}$ denotes the set of $n$-th order preimages of $c_t$, and $C_{\beta,i}\cap C_{\beta,j}=\varnothing, i\neq j$. One can check that $C_{\beta}$ is the preimages set of orbit $orb(c)$. $F_{\beta,n}$ denotes the number of preimages of $c_t$ with first hitting time equal to $n$.


The main result of the paper is the following Theorem.

\begin{theorem}\label{theob}
Let $D$ be the set of $\beta \in (1, 2]$ such that $f_\beta$ is a symmetric tent map with finite critical orbit. For $\beta \in D$, $f_{\beta}$ could be modified via Denjoy like surgery to be a $C^1$ unimodal map $\tilde{g}_{\beta}$ on $[0,1]$ admitting a thick hyperbolic repelling invariant Cantor set $\tilde{A}_\beta$, and $\tilde{g}_{\beta}$ contains a wild attractor $\tilde{A}_\beta^\ast\subset \tilde{A}_\beta$.
\end{theorem}

\begin{remark}
\noindent
\myproof
\begin{enumerate} [(1)]
\item The key of the surgery is to control the total length of inserted intervals. For $\beta \in D$, as we shall see in Proposition 1, there exists a constant $M_{\beta}$ such that  $F_{\beta, n} \leq M_{\beta} \beta^n$, which implies
    $$\sum\limits_{n=1}^{\infty}F_{\beta,n} a_{\beta,n}<+\infty.$$
    As a result, the total length of inserted intervals is finite.
\item According to \cite{de2012one} (page 222, Theorem 2.6), if $f: I\rightarrow I$ is $C^{1+\alpha}$ with $\alpha >0$, and $K\subset I$ is a compact, forward invariant, hyperbolic set for $f$, then $K=I$ or $K$ has Lebesgue measure equal to zero. Therefore, one should not expect that $\tilde{g}_{\beta}$ is $C^{1+\alpha}$ with $\alpha >0$.
\item $f_\beta$ admits finite critical orbit if and only if the critical point $c$ of $f_\beta$ is either periodic or eventually periodic. According to \cite{brucks2004topics}(page 158, Lemma 10.3.4), the set $D\cap (\sqrt{2},2]$ is dense in $(\sqrt{2},2]$. Notice that $f_\beta$ can be renormalized finite times via period-doubling renormalization if $\beta\in (1,\sqrt{2}]$, $D$ is also dense in $(1,2]$.
\item As we shall see in Section 3, $|\tilde{g}'_{\beta}(x)|=\beta >1$ as $x\in \tilde{A}_\beta$ and the Lebesgue measure of $\tilde{A}_\beta$ is positive. It follows that $\tilde{A}_{\beta}$ is a thick hyperbolic repelling  invariant Cantor set of $\widetilde{g}_{\beta}$, which contains a wild attractor $\tilde{A}^\ast_\beta$ with $h_{top}(\tilde{g}_{\beta}|_{\tilde{A}^\ast_\beta})=\log\beta >0$. $\tilde{A}_\beta^\ast$ attracts
    points in the Cantor set $\tilde{A}_\beta$ and countable set $\tilde{\Lambda}_\beta$, and the basin of attraction $\mathbb{B}(\tilde{A}_\beta^\ast)$ is nowhere dense.
\end{enumerate}
\end{remark}

For $\beta\in D$, we first construct a unimodal map $g_\beta$ over a finite interval by Denjoy like surgery, then transform $g_\beta$ to be a unit interval unimodal map $\tilde{g}_\beta$ which is topologically conjugated to $g_\beta$ by an affine diffeomorphism. More precisely, suppose $g_\beta$ is a map defined on interval $I_\beta:=[a_\beta,b_\beta]$ and $\varphi(x):[a_\beta,b_\beta]\rightarrow [0,1],x\mapsto \frac{x-a_\beta}{b_\beta-a_\beta}$, then put $\tilde{g}_\beta=\varphi\circ g_\beta \circ \varphi^{-1}$. In order to ensure the smoothness of $g_\beta$, we have to perform Denjoy like surgery on the tent map $f_\beta$ since we need to control the endpoint derivatives of inserted intervals to be equal to $\beta$ or $-\beta$. And we can only insert intervals in the preimages set of the periodic orbit of $c$, otherwise $g_\beta$ is not differentiable at the critical point $c=\frac{1}{2}$. The key of the surgery is to control the total length of the inserted intervals, which could be estimated by Perron-Frobenius theory as $\beta\in D$.

The remaining parts of the paper are organized as follows. In Section 2, we show two examples $f_\beta$ with $\beta=2$ and $\beta=(\sqrt{5}+1)/2=1.618\cdots$  to illustrate how to operate Denjoy like surgery  to construct  unimodal maps $g_\beta$ admitting thick hyperbolic repelling invariant Cantor sets. For the case $\beta=2$, the preimages of critical periodic orbit can be easily counted. The essential part is to show the smoothness of $g_\beta$. For the case $\beta=(\sqrt{5}+1)/2$, we illustrate how to count the preimages of the critical point $c=\frac{1}{2}$. In Section 3, we will construct $g_\beta$ for $\beta \in D$, and prove that it is a $C^1$ map admitting a thick hyperbolic repelling invariant Cantor set which contains a wild attractor.

\section{Two examples of $g_\beta$}
\setcounter{proposition}{0}
\renewcommand{\theproposition}{\arabic{proposition}}
\setcounter{theorem}{0}
\renewcommand{\thetheorem}{\arabic{theorem}}
\setcounter{lemma}{0}
\renewcommand{\thelemma}{\arabic{lemma}}
\noindent

In this section, considering the cases $\beta=2$ and $\beta=(\sqrt{5}+1)/2$, we will by Denjoy like surgery construct $C^1$ unimodal maps $\tilde{g}_\beta$ admitting thick hyperbolic repelling invariant Cantor sets which contain wild attractors. Before constructing $\tilde{g}_\beta$, we give two Lemmas.

\begin{lemma}\label{diffeomorphism of left}
Suppose $\frac{v_2-u_2}{v_1-u_1}\geq\beta>0$. Let
$$h(x):=u_2+\int_{u_1}^{x}
\beta+\frac{6[(v_2-u_2)-\beta(v_1-u_1)]}{(v_1-u_1)^3}(v_1-t)(t-u_1) dt,$$
$$r(x):=v_2+\int_{u_1}^{x}
-\beta+\frac{6[(v_2-u_2)-\beta(v_1-u_1)]}{(v_1-u_1)^3}(t-v_1)(t-u_1) dt.$$
Then $h(x)$ is a diffeomorphism from $[u_1,v_1]$ to $[u_2,v_2]$ with $h'(x)\geq\beta$ and $h'(u_1)=h'(v_1)=\beta$, and $r(x)$ is a diffeomorphism from $[u_1,v_1]$ to $[u_2,v_2]$ with $r'(x) \leq -\beta$ and $r'(u_1)=r'(v_1)=-\beta$.
\end{lemma}
\begin{proof}
\myproof
One can check that $h(u_1)=u_2, h(v_1)=v_2$, and
$$h'(x)=\beta+\frac{6[(v_2-u_2)-\beta(v_1-u_1)]}{(v_1-u_1)^3}(v_1-x)(x-u_1),$$
which follows that $h'(u_1)=h'(v_1)=\beta$.
Since $v_2-u_2\geq \beta(v_1-u_1)$, we have
$$\frac{6[(v_2-u_2)-\beta(v_1-u_1)]}{(v_1-u_1)^3}>0.$$
Thus $h'(x)$ obtains the minimum value at the endpoint $u_1$ or $v_1$, which follows that $h'(x)\geq\beta$ for each point $x\in [u_1, v_1]$, i.e., $h(x)$ is monotonically increasing. Therefore, $h(x)$ is a diffeomorphism.

The proof of $r(x)$ is the same as the proof of $h(x)$
$\hfill\square$
\end{proof}

\begin{lemma}\label{left}
Suppose $v_1-u_1=v_2-u_2=1, \ 1<\beta\leq 2$. Let
$$g(x):=u_2+\int_{u_1}^{x}\beta+6(\beta-1)(t-u_1)(t-v_1) dt,$$
$$w(x):=v_2+\int_{u_1}^{x}-\beta+6(1-\beta)(t-u_1)(t-v_1) dt.$$
Then $g(x)$ is a diffeomorphism from $[u_1,v_1]$ to $[u_2,v_2]$ with $g'(x)>0$ and $g'(u_1)=g'(v_1)=\beta$, and $w(x)$ is a diffeomorphism from $[u_1,v_1]$ to $[u_2,v_2]$ with $w'(x)<0$ and $w'(u_1)=w'(v_1)=-\beta$.
\end{lemma}

We label the point $x\in C_{\beta,n}, n=0,1,2,\cdots$ by a sequence $s(x)=s_n s_{n-1} \cdots s_0$ of 0's, c and 1's of length $n+1$, where $s_i$ is defined by
$$
s_i=\left \{
\begin{array}{ll}
0,& f_{\beta}^{n-i}(x)<c\\
c,& f_{\beta}^{n-i}(x)=c\\
1,& f_{\beta}^{n-i}(x)>c.
\end{array}
\right.
$$

Now we sketch the construction of $\tilde{g}_\beta$.

We insert an open interval at each point in $C_\beta$ to obtain a Cantor set $A_\beta$. In particular, if $0$ and $1$ belong to $C_\beta$ , we insert an interval of the form $[s,t)$ and $(s,t]$ at $0$ and $1$, respectively. More precisely, we insert an open interval with length 1 at each point $x\in orb(c)$ respectively. At each point $x\in C_{\beta,n}\setminus orb(c)$ labeled by the sequence $s(x)=s_n s_{n-1} \cdots s_0$, we insert an open interval
$J_{s_n s_{n-1} \cdots s_0}$ with length $a_{\beta,n}$, i.e., $|J_{s_n s_{n-1} \cdots s_0}|=a_{\beta,n}$. We denote the new interval by $I_\beta:=A_\beta\bigcup B_\beta$, where $B_\beta$ is the union of the intervals inserted over the unit interval $I$, and $A_\beta$ is the complement of $B_\beta$.

Define a $C^1$ diffeomorphism with derivatives of endpoints $\beta$ (or $-\beta$) on each inserted interval. For the interval $J_{s_n s_{n-1} \cdots s_0}$ inserted at $x\in C_{\beta,n} \backslash orb(c)$, we denote $J_{s_n s_{n-1} \cdots s_0}$ by $L(J_{s_n s_{n-1} \cdots s_0})$ if $x<c$, and $R(J_{s_n s_{n-1} \cdots s_0})$ if $x>c$. By Lemma \ref{diffeomorphism of left}, one can define a $C^1$ diffeomorphism $h_{s_n s_{n-1} \cdots s_0}: L(\bar{J}_{s_n s_{n-1} \cdots s_0})\rightarrow \bar{J}_{s_{n-1} s_{n-2} \cdots s_0}$ or a $C^1$ diffeomorphism $r_{s_n s_{n-1} \cdots s_0}: R(\bar{J}_{s_n s_{n-1} \cdots s_0})\rightarrow \bar{J}_{s_{n-1} s_{n-2} \cdots s_0}$, where $\bar{A}$ denotes the closure of $A$. Denote the intervals inserted at the $orb(c)=\{c, c_1, \cdots , c_t, \cdots , c_{t+m-1}\}$ as $J_0,J_1,\cdots, J_{t+m-1}$ respectively. Notice $f_\beta(c_{t+m-1})=c_t$, then we treat $J_{t+m}$ as $J_t$. By Lemma \ref{left}, one can define a $C^1$ diffeomorphism $g_i:\bar{J}_i\rightarrow \bar{J}_{i+1},1\leq i \leq t+m-1$ if $c_i<c$ and a $C^1$ diffeomorphism $w_i:\bar{J}_i\rightarrow \bar{J}_{i+1},1\leq i \leq t+m-1$ if $c_i>c$. One can also define a $C^1$ map $f:\bar{J}_0\rightarrow \bar{J}_1$ with unique turning point and $f'(p_0)=\beta,f'(q_0)=-\beta$, where $p_0,q_0$ are the left and right endpoint of interval $J_0$, respectively. We can extend the above maps to the interval $I_\beta$. Denote the map on the interval $I_\beta$ by $g_\beta$.

In fact, $A_\beta$ is a Cantor set with Lebesgue measure $m(A_\beta)=1$ since $C_\beta$ is a countable dense subset of $[0,1]$. Furthermore, one can check that $A_\beta$ is a hyperbolic repelling invariant Cantor set which contains a wild attractor of the dynamical system $(I_\beta,g_\beta)$. Notice that the inserted intervals except for countable points are attracted by the intervals inserted at the orbit $orb(c_t)$ and $A_\beta$ is attracted by itself.

If $|I_\beta|<+\infty$, denote $I_\beta=[a_\beta,b_\beta]$. Define $\varphi(x):[a_\beta,b_\beta]\rightarrow [0,1],x\mapsto \frac{x-a_\beta}{b_\beta-a_\beta}$ and $\tilde{g}_\beta:=\varphi\circ g_\beta \circ \varphi^{-1}$. $\tilde{g}_\beta$ and $g_\beta$ have the same topological properties since they are topologically conjugated. $\varphi$ and $\varphi^{-1}$ are $C^1$, then $\tilde{g}_\beta$ will be $C^1$ if $g_\beta$ is $C^1$. It is possible to make $g_\beta$ be a $C^1$ map by controlling the length of each inserted interval and its derivative at the endpoints of the interval.

In subsection \ref{subsection2.1}, we shall construct example $\tilde{g}_2$ and give a detailed proof of the smoothness of $\tilde{g}_2$. In subsection \ref{subsection2.2}, we present another example $\tilde{g}_\beta$ with $\beta=(\sqrt{5}+1)/2$ to illustrate how to count the preimages of critical point $c=\frac{1}{2}$.

\subsection{The case $\beta=2$}\label{subsection2.1}

One can check that $orb(c)=\{c,1,0\}$ and the point 0 is a fixed point of $f_\beta$ as $\beta=2$. The preimage tree of 0 as $\beta=2$, see Figure \ref{figure 1}.

\begin{figure}[h]
\centering
\includegraphics[scale=0.7]{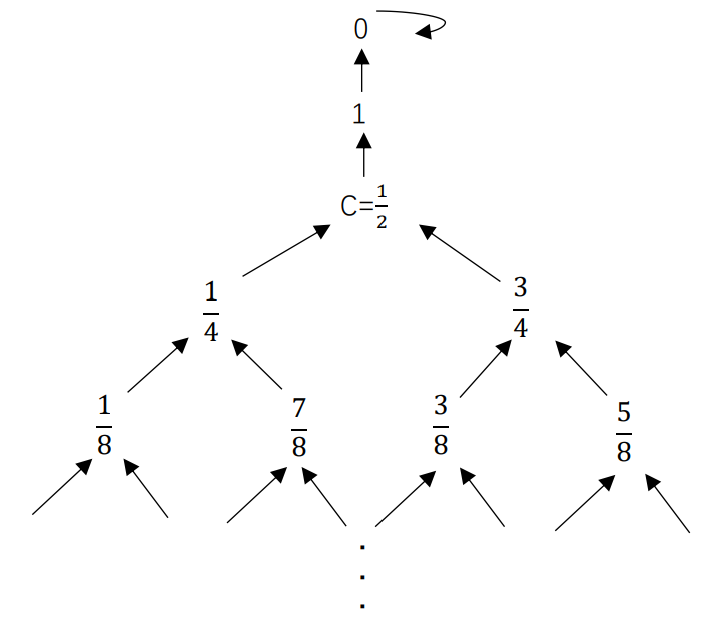}
\caption{The preimage tree of 0 as $\beta=2$.}
\label{figure 1}
\end{figure}

It is obviously seen that $F_{2,0}=F_{2,1}=1, F_{2,n}=2^{n-2},n\geq 2$ and
$$\sum\limits_{n=3}^{\infty}F_{2,n} a_{2,n}=\sum\limits_{n=3}^{\infty}2^{n-2} a_{2,n}=\sum\limits_{n=3}^{\infty}2^{n-2}\frac{1}{2^n}\frac{2}{n(n+1)}=\frac{1}{6}.$$
Hence the total length of the inserted intervals is $$1+1+1+\sum\limits_{n=3}^{\infty}F_{2,n} a_{2,n}=\frac{19}{6}.$$

In what follows, we will give the detailed procedure of the Denjoy like surgery.

Notice that $$a_{2,n}=\frac{1}{2^{n-1}}\frac{1}{n(n+1)},\quad \frac{a_{2,n-1}}{a_{2,n}}=\frac{2^{n-1}}{2^{n-2}}\frac{n(n+1)}{n(n-1)}=\frac{2(n+1)}{n-1}>2,$$
then $|J_{s_{n-1} s_{n-2} \cdots s_0}| \geq 2|J_{s_n s_{n-1} \cdots s_0}|$ if the point $x\notin orb(c)$ that labeled by the sequence $s(x)=s_{n-1} s_{n-2} \cdots s_0$. By Lemma \ref{diffeomorphism of left}, one can define a $C^1$ diffeomorphism $h_{s_n s_{n-1} \cdots s_0}: L(\bar{J}_{s_n s_{n-1} \cdots s_0})\rightarrow \bar{J}_{s_{n-1} s_{n-2} \cdots s_0}$ and a $C^1$ diffeomorphism $r_{s_n s_{n-1} \cdots s_0}: R(\bar{J}_{s_n s_{n-1} \cdots s_0})\rightarrow \bar{J}_{s_{n-1} s_{n-2} \cdots s_0}$. Denote the intervals inserted at the $orb(c)=\{c, 1, 0\}$ by $J_i=(p_i,q_i),i=0,1,2$ respectively. Notice that $q_0-p_0=q_1-p_1=1$, then one can define a map $f(x):[p_0,q_0]\rightarrow [p_1,q_1]$ as
$$
f(x)=\left \{
\begin{array}{ll}
f_1(x),& x\in [p_0,p_0+\frac{1}{2}]\\
f_2(x),& x\in (p_0+\frac{1}{2},q_0],
\end{array}
\right.
$$
where
$$f_1(x)=-8(x-p_0)^3+4(x-p_0)^2+2(x-p_0)+p_1,$$
$$f_2(x)=8(x-p_0-\frac{1}{2})^3-8(x-p_0-\frac{1}{2})^2+q_1.$$
One can check that $f(x)$ is a $C^1$ unimodal onto map with $f'(p_0)=2,f'(q_0)=-2$ and $f(p_0)=f(q_0)=p_1$.

By Lemma \ref{left}, one can define a $C^1$ diffeomorphism $w(x):\bar{J}_1\rightarrow \bar{J}_2$,  and a $C^1$ diffeomorphism $g(x):\bar{J}_2\rightarrow \bar{J}_2$. Define a map $g_2$ by
$$
g_2=\left \{
\begin{array}{ll}
h_{s_n s_{n-1} \cdots s_0}(x),& x\in L(\bar{J}_{s_n s_{n-1} \cdots s_0})\\
r_{s_n s_{n-1} \cdots s_0}(x),& x\in R(\bar{J}_{s_n s_{n-1} \cdots s_0})\\
f(x),& x\in \bar{J}_0\\
w(x),& x\in \bar{J}_1\\
g(x),& x\in \bar{J}_2,
\end{array}
\right.
$$
then we extend the map $g_2$ to the interval $I_2$, and the map on the interval $I_2$ is still denoted by $g_2$. Let $(e-,e+)$ be the interval inserted at the point $e$ such that $f_\beta^n(e)=0$ for some $n\geq 0$. The graph of map $g_2$ restricted to some intervals is given in Figure \ref{figure 2}.  The inserted intervals are attracted by an interior point $x_\ast$ except for a countable set $\cup_{i=0}^{+\infty}g_\beta^{-i}(0-)$, because $0-$ is a repelling fixed point. Notice that $0-$ does not belong to $A_2$ and $A_2$ is attracted by itself.

\begin{figure}[h]
\centering
\includegraphics[scale=0.7]{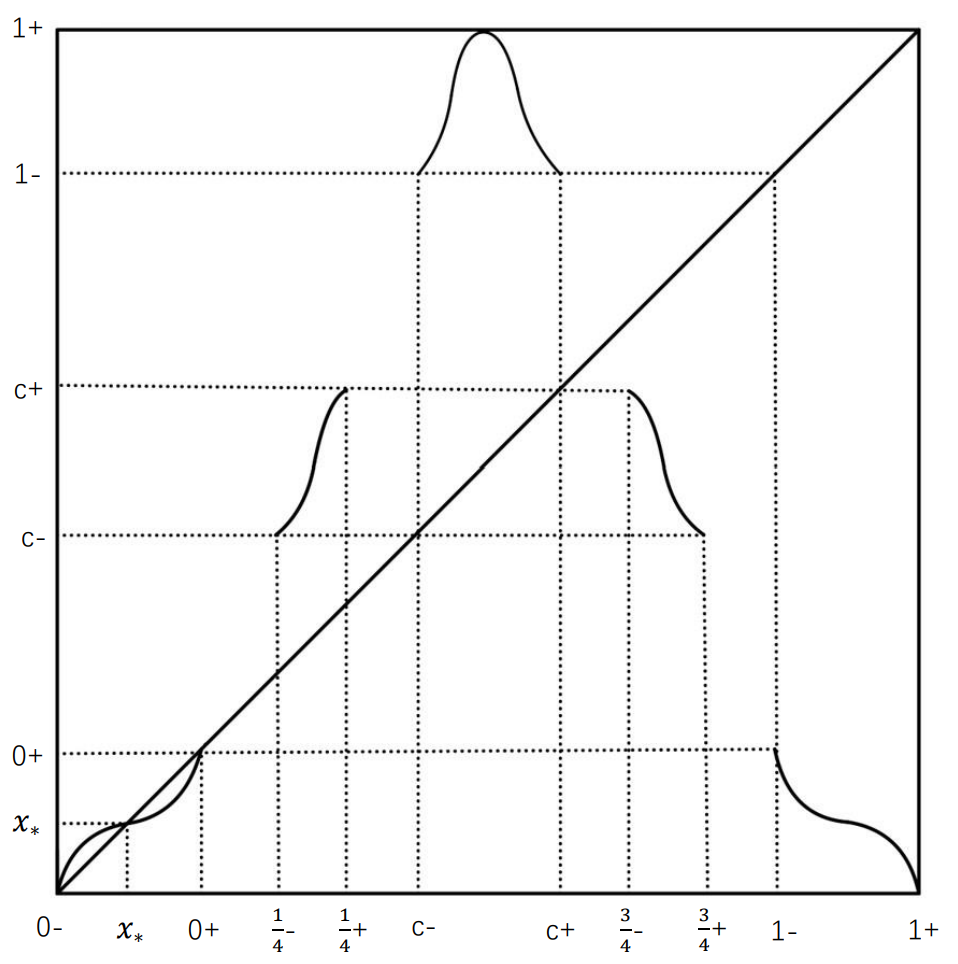}
\caption{The graph of map $g_2$ restricted to some intervals.}
\label{figure 2}
\end{figure}

Suppose $I_2=[a_2,b_2]$. Define an affine diffeomorphism $\varphi(x):[a_2,b_2]\rightarrow [0,1],x\mapsto \frac{x-a_2}{b_2-a_2}$ and $\tilde{g}_2:=\varphi\circ g_2 \circ \varphi^{-1}$.

\begin{theorem} \label{theorem example 1}
$\tilde{g}_2$ is a symmetric $C^1$  unimodal map on $[0,1]$ admitting a thick hyperbolic repelling invariant Cantor set which is also a wild attractor.
\end{theorem}
\begin{proof}
\myproof
We divide the proof in the following steps:

{\bf Step 1:} $A_2$ is a Cantor set with positive Lebesgue measure in $I_2$

As a countable set, $C_2$ has Lebesgue measure equal to zero, which follows that $A_2$ has Lebesgue measure equal to one. $A_2$ is a compact set since $A_2$ is the complement of an open set $B_2$. Notice that $C_2=0\bigcup\cup_{i=0}^{+\infty}\frac{1}{2^i}$, then $C_2$ is dense in $I$, which follows that $A_2$ is totally disconnected. Each point in $A_2$ can be approached by the endpoints of some inserted intervals, and the endpoints of those inserted intervals belong to $A_2$, thus each point $x\in A_2$ is an accumulation point in $A_2$, namely, $A_2\subseteq A'_2$. Then $A_2$ is a perfect set since $A_2$ is closed. Therefore $A_2$ is compact, totally disconnected, perfect, which follows that $A_2$ is a Cantor set.

{\bf Step 2:} $g_2$ is a $C^1$ map.

Let $c'$ be the critical point of $g_2$. Without loss of generality, we only consider the case $y<c'$.

We first show $g_2$ is differentiable. We have defined a $C^1$ map on each inserted interval $J_{s_n s_{n-1} \cdots s_0}$ such that the derivative is equal to 2 in the boundary of $J_{s_n s_{n-1} \cdots s_0}$. To prove that $g_2$ is differentiable, it is clearly enough to prove that $g_2$ is differentiable at each point $y\in A_2$ and that its derivative at $y$ is equal to 2. We show that
\begin{equation}\label{left derivative of 2}
\ \ \ \ \ \ \ \ \ \ \ \ \ \ \ \ \ \ \ \ \ \ \ \ \ \ \lim\limits_{z\downarrow y}\frac{m(g_2([y,z]))}{m([y,z])}=2,
\end{equation}
and similarly that the corresponding left-sided limit is equal to 2. Recall that $m(A)$ denotes the Lebesgue measure of $A$.

Let $X_1=[y,z]\cap B_2$ and $X_2=[y,z]\cap A_2$. One can check that $m(g_\beta(X_2))=2 m(X_2)$. In fact, $X_2$ can be divided into two parts $X_{21}$ and $X_{22}$, where $X_{22}$ denotes the endpoints of inserted intervals that intersect $X_2$, then $m(X_{22})=0$. Notice that corresponding $X_{21}$ on dynamical system $(I,f_\beta)$ we have $m(f_\beta(X_{21}))=2 m(X_{21})$, which follows that $m(g_\beta(X_2))=2 m(X_2)$. Therefore, to prove (\ref{left derivative of 2}), we only need to show
\begin{equation}\label{derivative of 2}
\ \ \ \ \ \ \ \ \ \ \ \ \ \ \ \ \ \ \ \ \ \ \ \ \ \ \lim\limits_{z\downarrow y}\frac{m(g_2(X_1))}{m(X_1)}=2.
\end{equation}

We shall adapt the idea in  \cite{de2012one}(Page 42, Theorem 2.3) to prove (\ref{derivative of 2}). Since the derivative of $g_2$ in the boundary of  any inserted interval is equal 2, (\ref{derivative of 2}) holds whenever a right-sided neighbourhood of $y$ is completely contained in one of the inserted intervals. Therefore we may assume that each right-sided neighbourhood of $y$ intersects an infinite number of inserted intervals. $m(X_1)$ can be estimated from the sum of Lebesgue measure of the inserted intervals completely contained in $[y,z]$ plus a piece of the inserted interval that contains $z$. Let
$$J(z)=\{J_{s_n s_{n-1} \cdots s_0}|J_{s_n s_{n-1} \cdots s_0}\cap[y,z]\neq \emptyset,n\in N\},$$
$$J'(z)=\{J_{s_n s_{n-1} \cdots s_0}|J_{s_n s_{n-1} \cdots s_0}\subset [y,z],n\in N\},$$
then
$$\sum\limits_{J_{s_n s_{n-1} \cdots s_0}\in J'(z)}a_{2,n}\leq m(X_1)\leq \sum\limits_{J_{s_n s_{n-1} \cdots s_0}\in J(z)}a_{2,n}.$$
Since $J_{s_n s_{n-1} \cdots s_0}$ is contained in $[y,z]$ if and only if $J_{s_{n-1} s_{n-2} \cdots s_0}$ is contained in $g_2([y,z])$, we get
$$\sum\limits_{J_{s_n s_{n-1} \cdots s_0}\in J'(z)}a_{2,n-1}\leq m(g_2(X_1))\leq \sum\limits_{J_{s_n s_{n-1} \cdots s_0}\in J(z)}a_{2,n-1}.$$
Therefore
\begin{equation}\label{diffeomorphism of 2}
\frac{\sum\limits_{J_{s_n s_{n-1} \cdots s_0}\in J'(z)}a_{2,n-1}}{\sum\limits_{J_{s_n s_{n-1} \cdots s_0}\in J(z)}a_{2,n}}\leq \frac{m(g_2(X_1))}{m(X_1)}\leq\frac{\sum\limits_{J_{s_n s_{n-1} \cdots s_0}\in J(z)}a_{2,n-1}}{\sum\limits_{J_{s_n s_{n-1} \cdots s_0}\in J'(z)}a_{2,n}}.
\end{equation}
By construction, $J'(z)\subset J(z)$, $\sharp\{J(z)\backslash J'(z)\}\leq 1$ and $\inf\{n|J_{s_n s_{n-1} \cdots s_0}\in J(z)\}\rightarrow +\infty, \inf\{n|J_{s_n s_{n-1} \cdots s_0}\in J'(z)\}\rightarrow +\infty$ as $z\rightarrow y$. Since $a_{2,n}\rightarrow 0$ and $a_{2,n-1}/a_{2,n}\rightarrow 2$ as $n\rightarrow \infty$, we have
\[\lim\limits_{z\downarrow y}\frac{\sum\limits_{J_{s_n s_{n-1} \cdots s_0}\in J'(z)}a_{2,n-1}}{\sum\limits_{J_{s_n s_{n-1} \cdots s_0}\in J(z)}a_{2,n}}=\lim\limits_{z\downarrow y}\frac{\sum\limits_{J_{s_n s_{n-1} \cdots s_0}\in J'(z)}a_{2,n-1}}{\sum\limits_{J_{s_n s_{n-1} \cdots s_0}\in J'(z)}a_{2,n}}=\lim\limits_{n_1\rightarrow +\infty \atop n_1<n_2<\cdots}\frac{\sum\limits_{i=1}^\infty a_{2,n_i-1}}{\sum\limits_{i=1}^\infty a_{2,n_i}}\]
\[=\lim\limits_{n_1\rightarrow +\infty \atop n_1<n_2<\cdots}\frac{\sum\limits_{i=1}^\infty \frac{1}{2^{n_i-2}}\frac{1}{n_i(n_i-1)}}{\sum\limits_{i=1}^\infty \frac{1}{2^{n_i-1}}\frac{1}{n_i(n_i+1)}}=\lim\limits_{n_1\rightarrow +\infty \atop n_1<n_2<\cdots}2\frac{\sum\limits_{i=1}^\infty \frac{1}{2^{n_i}}\frac{1}{n_i^2}}{\sum\limits_{i=1}^\infty \frac{1}{2^{n_i}}\frac{1}{n_i^2}}=2.\]
Similarly, the limit of the right side of (\ref{diffeomorphism of 2}) is equal to 2. It follows that (\ref{derivative of 2}) holds.

Next, we shall show $g'_2$ is continuous. Notice that $g'_2$ is continuous on $B_2$, and $g'_2(y)=2$ as $y\in A_2$. To prove $g'_2$ is continuous at each point $y\in A_2$, we only need to show that
\begin{equation}\label{conti of deri 2}
\ \ \ \ \ \ \ \ \ \ \ \ \ \ \ \ \ \ \ \ \ \ \ \ \ \ \lim\limits_{z\downarrow y}g'_2(z)=g'_2(y)=2.
\end{equation}
Since $\lim\limits_{z\downarrow y \atop z\in A_2}g'_2(z)=2$, to prove (\ref{conti of deri 2}), it is sufficient to show
\begin{equation}\label{continu 2}
\ \ \ \ \ \ \ \ \ \ \ \ \ \ \ \ \ \ \ \ \ \ \ \ \ \ \lim\limits_{z\downarrow y \atop z\in B_2}g'_2(z)=2.
\end{equation}
(\ref{continu 2}) holds whenever a right-sided neighbourhood of $y$ is completely contained in one of the inserted intervals. Therefore we may assume that each right-sided neighbourhood of $y$ is intersects an infinite number of inserted intervals. Same as before, we have
$$\inf\{n|J_{s_n s_{n-1} \cdots s_0}\in J(z)\}\rightarrow +\infty \quad  \quad (z\rightarrow y).$$
Put $J_{s_{n} s_{n-1} \cdots s_0}=[u_1,v_1],J_{s_{n-1} s_{n-2} \cdots s_0}=[u_2,v_2]$, then
$$v_1-u_1=\frac{n-1}{2(n+1)}(v_2-u_2)$$
and
$$2 \leq h'_{s_{n} s_{n-1} \cdots s_0}\leq 2+\frac{6[(v_2-u_2)-2(v_1-u_1)]}{(v_1-u_1)^3}(v_1-\frac{u_1+v_1}{2})(\frac{u_1+v_1}{2}-u_1)=2+\frac{6}{n-1}.$$
Therefore,
$$\lim\limits_{n\rightarrow +\infty}h'_{s_{n} s_{n-1} \cdots s_0}(x)=2, \quad x\in J_{s_{n} s_{n-1} \cdots s_0}.$$
It follows that (\ref{continu 2}) holds.

{\bf Step 3:} $A_2$ is a thick hyperbolic repelling invariant Cantor set and also a wild attractor.

According to \cite{de2012one} (page 213, Lemma 2.1), a compact invariant set $K$ of a $C^1$ map $f$ is a hyperbolic set if and only if for each $x\in K$ there exists an integer $n=n(x)$ such that $|Df^n(x)|>1$. Notice that $g_2(A_2)\subset A_2$ follows from $g_2^{-1}(B_2)\subset B_2$, then $A_2$ is a hyperbolic repelling invariant Cantor set since $|g'_2(x)|=2>1$ for each point $x\in A_2$. One can check that $A_2$ is also a wild attractor. In fact, the basin of attraction $\mathbb{B}(A_2)$ is exactly $A_2$ and there exists no forward invariant closed set strictly contained in $A_2$ has this property since $g_2|A_2$ is topologically transitive, which implies that $A_2$ is a metric attractor. There exists no open interval attracted by $A_2$ since $A_2$ is nowhere dense in $I_2$, which implies that $A_2$ is not topological attractor.

{\bf Step 4:} Affine diffeomorphism.

$\tilde{g}_2$ is a $C^1$ map since $\varphi$ and $\varphi^{-1}$ are $C^1$. In fact,
\begin{eqnarray*}
(\tilde{g}_2)'(x) &=(\varphi\circ g_2 \circ \varphi^{-1})'(x)\\
 &= (\varphi)'(g_2 \circ \varphi^{-1}(x)) \cdot (g_2)'(\varphi^{-1}(x)) \cdot(\varphi^{-1})'(x)\\
 &=\frac{1}{b-a}\cdot(g_2)'(\varphi^{-1}(x))\cdot (b-a)\\
 &=(g_2)'(\varphi^{-1}(x))
\end{eqnarray*}
One can check that for each $x\in \varphi(A_2)$, we have
$|(\tilde{g}_2)'(x)|=2>1$. $\tilde{g}_2$ is also a unit interval symmetric unimodal map, which admits the same topological properties as $g_2$ since $\tilde{g}_2$ and $g_2$ are topologically conjugated. Then $\tilde{g}_2$ has a hyperbolic repelling invariant Cantor set $\varphi(A_2)$ which is also a wild attractor. Furthermore, the hyperbolic repelling invariant Cantor set of $\tilde{g}_2$ is also thick since the total length of inserted intervals is finite.

The proof of Theorem 1 is complete.
$\hfill\square$
\end{proof}

The key step to the construction of $\tilde{g}_2$ is to count the number of preimage sets of the periodic point $0$. The case $\beta=2$ is so special. It is easy to see that the number of $n$th-order preimages of the periodic point $f_\beta^2(\frac{1}{2})=0$ is $2^{n-2},n\geq 2$. We will give another example where the number of preimages set of the periodic point $c_t$ is not simple.

\subsection{The case $\beta={(\sqrt{5}+1)}/{2}$}\label{subsection2.2}

Let $f:I\rightarrow I$ be unimodal map with critical point $c_0$. For each $x\in I$, the itinerary of $x$ is defined by
$$I(x)=I_0I_1I_2 \cdots,$$
where
$$
I_i=\left \{
\begin{array}{ll}
0,& f^i(x)<c_0\\
\ast,& f^i(x)=c_0\\
1,& f^i(x)>c_0.
\end{array}
\right.
$$
The parity-lexicographical ordering $\preceq$ works as follows: Let $u\neq v$ be different itineraries and find the first position where $u, \ v$ differ. We compare in that position by the order $0<\ast<1$ if the number of 1's precedding this position is even and by the order $1<\ast<0$ otherwise. If $x, y\in I$ and $x<y$, then $I(x)\preceq I(y)$.

\begin{lemma}\label{fibonacci}
Let $f_{\beta}$ is a symmetric tent map with $\beta=(\sqrt{5}+1)/2$, we have
$$\sum\limits_{n=1}^{\infty}F_{\beta,n} a_{\beta,n}<\infty.$$
\end{lemma}

\begin{proof}
\myproof
It is easy to verify that $c$ is a $3$-period point of $f_{\beta}$, and its orbit is $orb(c)=\{c,c_1,c_2\}$, where $c_i=f^i(c), i=1,2$. In fact, $f_\beta(\frac{1}{2})=\frac{\sqrt{5}+1}{4},f_\beta^2(\frac{1}{2})=\frac{\sqrt{5}-1}{4},f_\beta^3(\frac{1}{2})=\frac{1}{2}$.
One can check that itineraries of $c,c_1,c_2$ are $I(c)=(\ast 1 0)^\infty, \ I(c_1)=(1 0\ast)^\infty, \ I(c_2)=(0\ast 1)^\infty$ respectively.

For $x\in I$, $x$ admits two preimages if $I(x) < I(c_1)$ and no preimage if $I(x)>I(c_1)$. We only need to compare the first three digits of $I(x)$ and $I(c_1)$ since $I(c_1)$ is $3$-periodic. Notice that $c_1$ admits unique preimage $c$.

Let $c_n(k_1k_2k_3),\ k_i \in \{0,1\}, i=1,2,3$ denote the number of element in $C_{\beta,n}$, whose itinerary admits first three digits $k_1k_2k_3$. Recall that
$$C_{\beta,n}=\{x|f_{\beta}^n(x)=c,\ f_{\beta}^s(x)\neq c, s<n, \  x\in I\}.$$
We observe that if the first three digits of $I(x)$ are 100, then $x$ admits no preimage. Other observations see Table \ref{table 1}.
\begin{table}[h]
\centering
\includegraphics[scale=0.8]{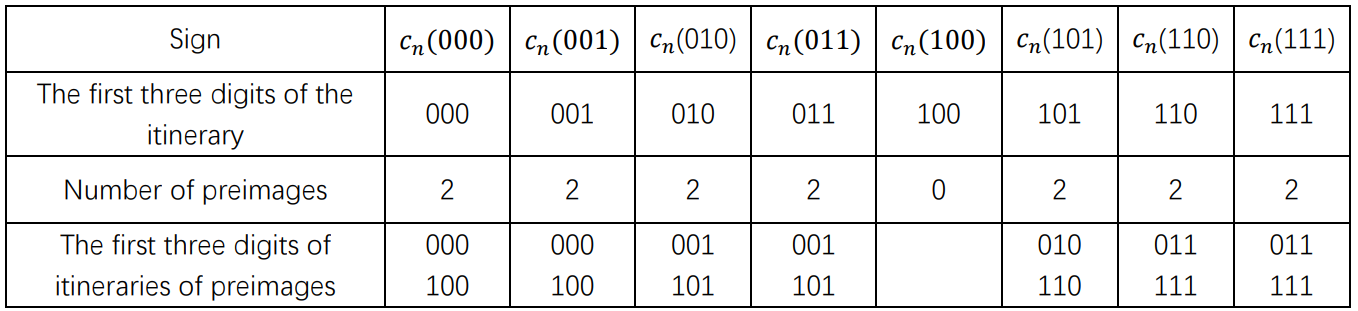}
\caption{The first three digits of itineraries of preimages}
\label{table 1}
\end{table}

Recall that $F_{\beta,n}=\sharp\{C_{\beta,n}\}.$ By direct calculation, we have $F_{\beta,1}=2, F_{\beta,2}=4, F_{\beta,3}=6$. The preimage tree of c, see Figure \ref{figure 3}.
\begin{figure}[h]
\centering
\includegraphics[scale=0.7]{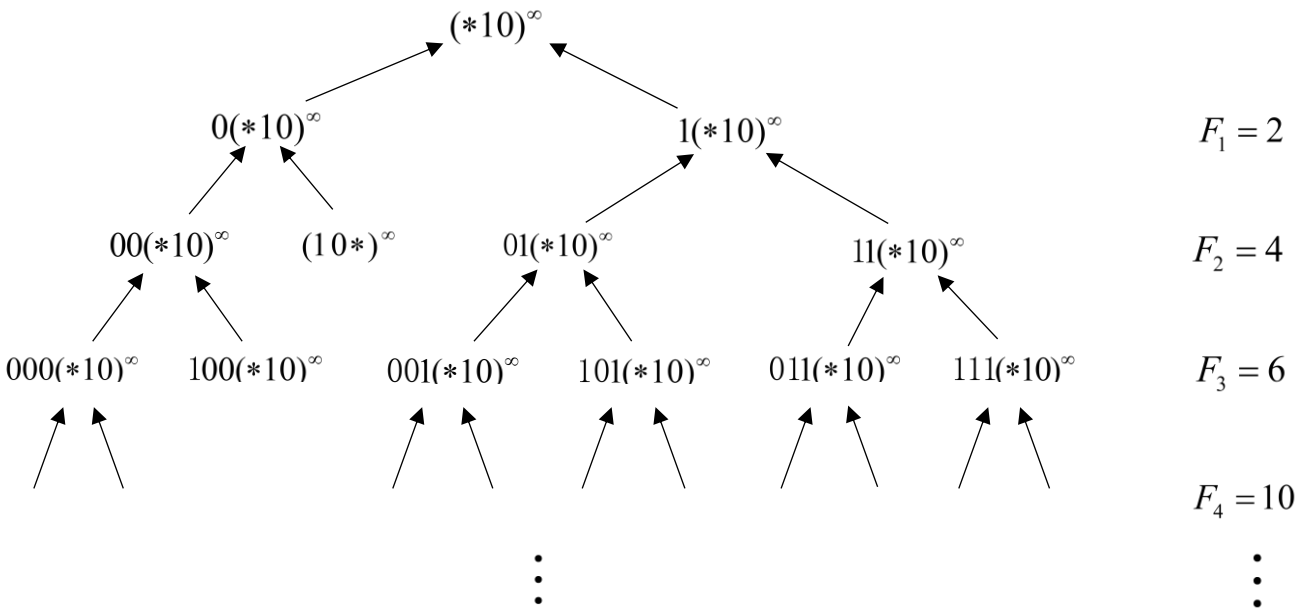}
\caption{The preimage tree of c as $\beta=(\sqrt{5}+1)/2$.}
\label{figure 3}
\end{figure}

By the table \ref{table 1}, we have
\begin{eqnarray*}
c_{n+1}(000)&=c_{n+1}(100)=c_n(000)+c_n(001),\\
c_{n+1}(001)&=c_{n+1}(101)=c_n(010)+c_n(011),\\
c_{n+1}(010)&=c_{n+1}(110)=c_n(101),\\
c_{n+1}(011)&=c_{n+1}(111)=c_n(110)+c_n(111).
\end{eqnarray*}
It follows that
\begin{eqnarray*}
F_{\beta,n} &=2(c_{n+1}(000)+c_{n+1}(001)+c_{n+1}(010)+c_{n+1}(011))\\
 & =2(c_{n+1}(100)+c_{n+1}(101)+c_{n+1}(110)+c_{n+1}(111)),\quad n\geq 3.
\end{eqnarray*}

We shall show $F_{\beta,n}=F_{\beta,n-1}+F_{\beta,n-2}, n\geq 3$ by induction. Assume $F_{\beta,n}=F_{\beta,n-1}+F_{\beta,n-2}$ holds if $3\leq n\leq k$. Since
\begin{eqnarray*}
F_{\beta,n-1}+F_{\beta,n-2} &= 2(c_{n-1}(000)+c_{n-1}(001)+c_{n-1}(010)+c_{n-1}(011))\\
&\ \ \ +2(c_{n-2}(000)+c_{n-2}(001)+c_{n-2}(010)+c_{n-2}(011))\\
&=2(c_{n}(000)+c_{n}(001))+2(c_{n-1}(000)+c_{n-1}(001))\\
&=2(c_{n}(000)+c_{n}(001))+2(c_{n}(000))\\
&=2(2c_{n}(000)+c_{n}(001)),
\end{eqnarray*}
then $F_{\beta,n}=2(c_{n}(000)+c_{n}(001)+c_{n}(010)+c_{n}(011))=2(2c_{n}(000)+c_{n}(001))$, namely, $c_{n}(000)=c_{n}(010)+c_{n}(011)=c_{n}(101), \ 3\leq n\leq k$.

If $n=k+1$, we have
\begin{eqnarray*}
F_{\beta,k+1} &=2(c_{n+1}(000)+c_{n+1}(001)+c_{n+1}(010)+c_{n+1}(011))\\ &=2(c_{n}(000)+c_{n}(001)+c_{n}(010)+c_{n}(011))\\
& \ \ \ +2(c_{n}(101)+c_{n}(110)+c_{n}(111))\\
&=F_{\beta,k}+2(c_{n-1}(000)+c_{n-1}(101)+c_{n-1}(110)+c_{n-1}(111))\\
&=F_{\beta,k}+2(c_{n-1}(100)+c_{n-1}(101)+c_{n-1}(110)+c_{n-1}(111))\\
&=F_{\beta,k}+F_{\beta,k-1}.
\end{eqnarray*}

Furthermore, we have
$$F_{\beta,n}=k_1(\frac{1-\sqrt{5}}{2})^n+k_2(\frac{1+\sqrt{5}}{2})^n, \quad n\geq 1,$$
where $k_1=\frac{5-\sqrt{5}}{5}, k_2=\frac{5+\sqrt{5}}{5}$.
Then
\begin{eqnarray*}
\sum\limits_{n=1}^{\infty}F_{\beta,n} a_{\beta,n} &=\sum\limits_{n=1}^{\infty}[k_1(\frac{1-\sqrt{5}}{2})^n+k_2(\frac{1+\sqrt{5}}{2})^n](\frac{2}{1+\sqrt{5}})^n\frac{2}{n(n+1)}\\
&=\sum\limits_{n=1}^{\infty}[k_1(-1)^n(\frac{3-\sqrt{5}}{2})^n+k_2]\frac{2}{n(n+1)}\\
&<\sum\limits_{n=1}^{\infty}\frac{2(k_1+k_2)}{n(n+1)}=4<\infty.
\end{eqnarray*}
The proof is complete.
$\hfill\square$
\end{proof}

Similar to the construction of $g_2$, one can construct a $C^1$ unimodal map $g_\beta$ on $I_\beta=A_\beta\cup B_\beta$ with $\beta={(1+\sqrt{5})}/{2}$. Recall that $B_\beta$ is the union of the intervals inserted over the unit interval $I$ and $A_\beta$ is the complement of $B_\beta$. We also denote the critical point of $g_\beta$ by $c'$. Let $(c-,c+)$ and $(c_i-,c_i+),i=1,2$ be the intervals inserted at points $c,c_1,c_2$, respectively. The graph of map $g_\beta$ restricted to some intervals is given in Figure \ref{figure 4}.

One can check that the set $B_\beta$ except for countable points eventually enters the intervals inserted at the orbit of $c$ after countable iterations of $g_\beta$, and $A_\beta$ is attracted by itself. Notice that $c'\in (c-,c+), g_\beta(c')=c_1+\in A_\beta$ and $orb(c_1+)=\{c_1+,c_2-,c-,c_1-,c_2+,c+\}$, where $c_1-$ is a $3$-period point of $g_\beta$. The preimages set $\cup_{i=0}^{+\infty}g_\beta^{-i}(c')$ is attracted by the set $\{c_1-,c_2+,c+\}$ in $A_\beta$.

\begin{figure}[h]
\centering
\includegraphics[scale=0.6]{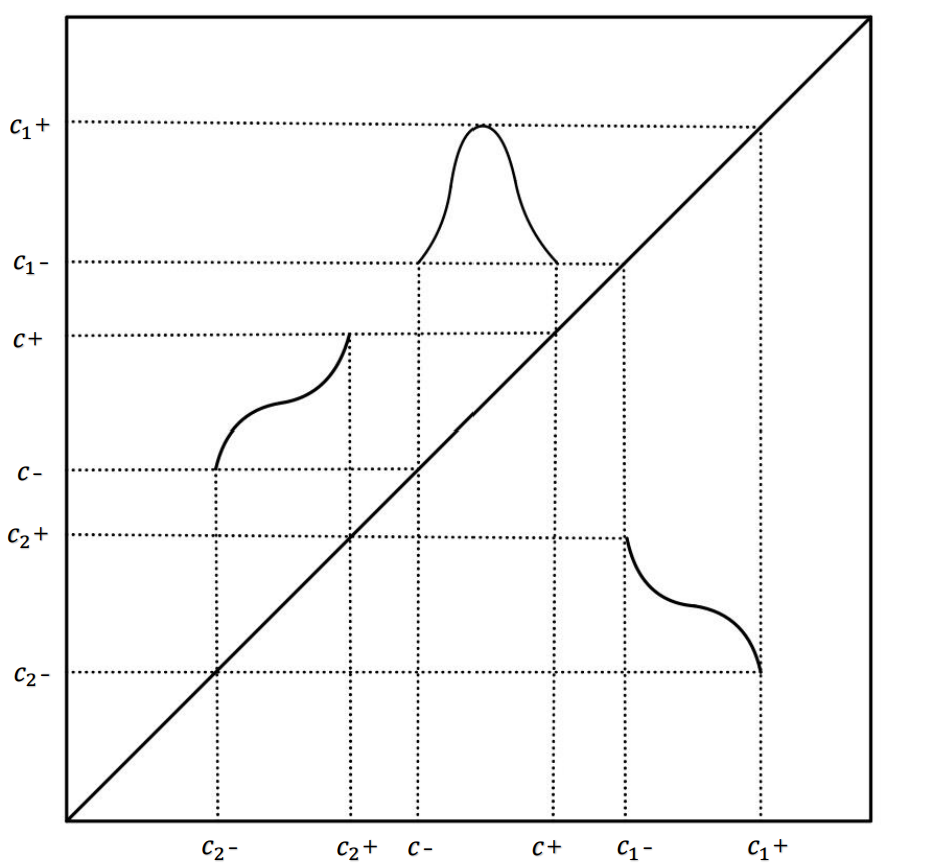}
\caption{The graph of map $g_\beta$ with $\beta={(1+\sqrt{5})}/{2}$ restricted to some intervals.}
\label{figure 4}
\end{figure}

Same as the case $\beta=2$, suppose $I_\beta=[a_\beta,b_\beta]$ and define $\varphi(x):[a_\beta,b_\beta]\rightarrow [0,1],x\mapsto \frac{x-a_\beta}{b_\beta-a_\beta}$ and $\tilde{g}_\beta:=\varphi\circ g_\beta \circ \varphi^{-1}$ as $\beta={(1+\sqrt{5})}/{2}$.
\begin{theorem}
Suppose . $\tilde{g}_\beta$ is a $C^1$ unimodal map on $[0,1]$ admitting thick hyperbolic repelling invariant Cantor set which contains a wild attractor.
\end{theorem}
\begin{proof}
\myproof
Same as the unimodal map $g_2$, one can check that $g_\beta$ is a $C^1$ unimodal map admitting hyperbolic repelling invariant Cantor set $A_\beta$ since $g_\beta(A_\beta)\subset A_\beta$ and $|g'_\beta(x)|=\beta>1, x\in A_\beta$ . Notice that $f_\beta$ is not renormalizable as $\beta={(1+\sqrt{5})}/{2}$, which follows that $g_\beta$ is not renormalizable . Let $\tilde{P}_\beta$ be the core on the system $(I_\beta,g_\beta)$ corresponding to the core $P_\beta$ of the system $(I,f_\beta)$. Define $A_\beta^\ast:=A_\beta\cap \tilde{P}_\beta$, then $g_\beta(A_\beta^\ast)\subset A_\beta^\ast$ since $g_\beta(\tilde{P}_\beta)=\tilde{P}_\beta$ and $g_\beta(A_\beta)\subset A_\beta$. Furthermore, one can check that $A_\beta^\ast$ is a wild attractor of $g_\beta$ since $A_\beta^\ast$ is not decomposable and there exists no open interval attracted by $A_\beta^\ast$. Similar to the proof of Theorem \ref{theorem example 1}, $\tilde{g}_\beta$ is a unit interval $C^1$ unimodal map admitting hyperbolic repelling invariant Cantor set which contains a wild attractor.

Notice that the length of inserted intervals differs by a finite number from $\sum\limits_{n=1}^{\infty}F_{\beta,n} a_{\beta,n}.$ By Lemma \ref{fibonacci},
the total length of inserted intervals is finite, which follows that the hyperbolic repelling invariant Cantor set of $\tilde{g}_\beta$ is thick.
$\hfill\square$
\end{proof}

\section{The general case $\beta\in D$}
\renewcommand{\thetheorem}{\Alph{theorem}}
\noindent

In this section, by operating Denjoy like surgery namely inserting intervals at the preimages $\cup_{i=0}^{+\infty}f_\beta^{-i}(c_t)$ of $c_t$, we shall construct a map $g_{\beta}$ for each $\beta\in D$, where $D\subset(1,2]$ such that $f_\beta$ is a symmetric tent map with finite critical orbit. One of the keys of the surgery is to control the total length of the inserted intervals so that it is finite, in fact, which could be ensured by Perron-Frobenius theory as $\beta\in D$.

{\bf Perron-Frobenius Theorem}\cite{walters2000introduction}. Let $A$ be a non-negative square  matrix.
\begin{enumerate}[(1)]
\item There is a non-negative eigenvalue $\lambda$ such that no eigenvalue of $A$ has absolute value greater than $\lambda$.
\item Corresponding to the eigenvalue $\lambda$ there is a non-negative left eigenvector $w$ and a non-negative right eigenvector $v$ which are also called left Perron vector and right Perron vector, i.e., there exist non-negative column vectors $v$ and $w$ such that $Av=\lambda v$ and $w^TA=\lambda w^T$.
\item If $A$ is irreducible then $\lambda$ is a simple eigenvalue and the corresponding eigenvectors are strictly positive.
\end{enumerate}

Recall that a square real matrix $A$ is nonnegative if each entry is nonnegative, irreducible if for each $i,j$ corresponding to an entry in the matrix, there exists a positive integer $k$ such that $(A^k)_{ij}>0$. Choose $w$ and $v$ such that $v^Tw=1$, where $v$ and $w$ is the right and left Perron vectors of a nonnegative matrix $A$ with spectral radius $\lambda_A$, then we have
$$\lim\limits_{n\rightarrow +\infty}(\frac{1}{\lambda_A}A)^n=vw^T.$$

\begin{proposition}\label{summable} Let $D$ be the set of $\beta \in (1, 2]$ such that $f_\beta$ with finite critical orbit, then $D$ is dense in $(1,2]$, and for each $\beta \in D$, we have
$$\sum\limits_{n=1}^{\infty}F_{\beta,n} a_{\beta,n}<+\infty.$$
\end{proposition}
\begin{proof}
\myproof
By (3) of Remark 1, $D$ is dense in $(1,2]$. If the orbit of the critical point $c$ is finite, then there exists a subshift of finite type  corresponding to it. We refer to the matrix $B$ as the transition matrix for the subshift of finite type. Suppose $B$ is a $s\times s$ square matrix, then there exists $s$ blocks which represent $s$ states. According to \cite{lind2021introduction}, the spectral radius of $B$ is precisely equal to $\beta$ since the topological entropy of $f_\beta$ is equal to $\log \beta$. Note that $B^n$ counts the numbers of paths of length $n$ in the transition graph $G_B$ corresponding to $B$, i.e., between vertices $i$ and $j$ there are exactly $B^n_{i,j}$ paths of length $n$ in $G_B$. Suppose the periodic point $c_t$ belongs to the $m$-th block, we can conclude that $\sum\limits_{i=1}^{s}B^n_{i,m}$ is precisely the number of the $n$-th order preimages of $c_t$. In fact, $B^n_{i,m}$ counts the number of the $n$-th order preimages of $c_t$ that belongs to the $i$-th block. So we have $F_{\beta,n}\leq\sum\limits_{i=1}^{s}B^n_{i,m}$. Since $\lim\limits_{n\rightarrow +\infty}\frac{B^n}{\beta^n}=vw^T$, where $w$ and $v$ are the left Perron vector and the right Perron vector of $B$, respectively, then there exist positive number $M$ and $N$ such that $\sum\limits_{i=1}^{s}B^n_{i,m}\leq M\beta^n$ as $n\geq N$ which implies that $\sum\limits_{n=1}^{\infty}F_{\beta,n} a_{\beta,n}<+\infty$.
$\hfill\square$
\end{proof}

We sketch the proof of Theorem A.

We will first construct a $C^1$ unimodal map $g_\beta:I_\beta\rightarrow I_\beta$ for each $\beta\in D$ by inserting suitable intervals at the preimages of the critical periodic orbit of $f_\beta$ and defining suitable mappings on them. The idea is based on the construction of Denjoy. The proof of smoothness of $g_\beta$ is similar to the proof of Theorem 2.3 in \cite{de2012one}. We present that $g_\beta$ admits a thick hyperbolic repelling invariant Cantor set $A_\beta$ which contains a wild attractor $A_\beta^\ast$, and divide the proof into two cases: $\beta\in(\sqrt{2}, 2]\cap D$ and $\beta\in(1,\sqrt{2}]\cap D$.

In the first case, $f_\beta$ is not renormalizable which implies $f_\beta$ is topologically transitive in the core $P_\beta$. It is easy to see that $A_\beta^\ast=A_\beta \cap \tilde{P}_\beta$, where $\tilde{P}_\beta$ is the core of $g_\beta$. In the second case, $f_\beta$ can be renormalized finite times via period-doubling renormalization. If $\sqrt{2}< \beta^m \leq 2$ for some $m:=2^k, k\geq 1$, then $f_\beta$ is $k$ times renormalizable, which implies that $g_\beta$ is $k$ times renormalizable and there exists a restrictive interval $\tilde{R}_\beta^k$ such that $g_\beta^{2^k}\mid \tilde{R}_\beta^k$ is again a unimodal map. Denote the core of $g_\beta^{2^k}\mid \tilde{R}_\beta^k$ by $\tilde{P}_\beta^k$, then one can check that $A_\beta^\ast=A_\beta\cap (\cap_{i=0}^{2k-1} g_\beta^i(\tilde{P}_\beta^k))$.

Next we consider an affine diffeomorphism. Suppose $I_\beta=[a_\beta,b_\beta]$. Let $\varphi(x):[a_\beta,b_\beta]\rightarrow [0,1],x\mapsto \frac{x-a_\beta}{b_\beta-a_\beta}$ and $\tilde{g}_\beta=\varphi\circ g_\beta \circ \varphi^{-1}$, then $\tilde{g}_\beta$ is a unit interval $C^1$ unimodal map admitting a thick hyperbolic repelling invariant Cantor set which contains a wild attractor.

\vspace{0.5cm}
\noindent
{\bf{The Proof of Theorem A.}}
\begin{proof}
\myproof
We divide our proof in the following steps:

{\bf Step 1:} Construct a unimodal $g_\beta$ by operating Denjoy like surgery on tent map $f_\beta$ for each $\beta\in D$.

Notice that $|J_{s_{n-1} s_{n-2} \cdots s_0}| \geq \beta|J_{s_n s_{n-1} \cdots s_0}|$ if the point $x\notin orb(c)$ labeled by the sequence $s(x)=s_{n-1} s_{n-2} \cdots s_0$. Similar to the construction of map $g_2$, By Lemma \ref{diffeomorphism of left}, one can define a $C^1$ diffeomorphism $h_{s_n s_{n-1} \cdots s_0}: L(\bar{J}_{s_n s_{n-1} \cdots s_0})\rightarrow \bar{J}_{s_{n-1} s_{n-2} \cdots s_0}$  whose derivative of the endpoints is equal to $\beta$, and a $C^1$ diffeomorphism $r_{s_n s_{n-1} \cdots s_0}: R(\bar{J}_{s_n s_{n-1} \cdots s_0})\rightarrow \bar{J}_{s_{n-1} s_{n-2} \cdots s_0}$ whose derivative of the endpoints is equal to $-\beta$.

Let $J_0,J_1,\cdots, J_{t+m-1}$ be the intervals inserted at the orbit $orb(c)$ respectively, where $c_t$ is $m$-periodic. Denote the interval $J_0$ by $(p_0,q_0)$ and the interval $J_1$ by $(p_1,q_1)$. Notice that $q_0-p_0=q_1-p_1=1$, then one can define a map $f(x):[p_0,q_0]\rightarrow [p_1,q_1]$ as
$$
f(x)=\left \{
\begin{array}{ll}
f_1(x),& x\in [p_0,p_0+\frac{1}{2}]\\
f_2(x),& x\in (p_0+\frac{1}{2},q_0],
\end{array}
\right.
$$
where
$$f_1(x)=(-16+4\beta)(x-p_0)^3+(-4\beta+12)(x-p_0)^2+\beta(x-p_0)+p_1,$$
$$f_2(x)=(16-4\beta)(x-p_0-\frac{1}{2})^3+(2\beta-12)(x-p_0-\frac{1}{2})^2+q_1.$$
One can check that $f(x)$ is a $C^1$ onto map with a critical point $c'$ and satisfies the following conditions:
$$f(p_0)=f(q_0)=p_1, \ f'(p_0)=\beta, \ f'(q_0)=-\beta.$$

Notice $f_\beta(c_{t+m-1})=c_t$, then we treat $J_{t+m}$ as $J_t$. By Lemma \ref{left}, one can define a $C^1$ diffeomorphism $g_i:\bar{J}_i\rightarrow \bar{J}_{i+1},1\leq i \leq t+m-1$ if $c_i<c$, whose derivative of the endpoints is equal to $\beta$, and a $C^1$ diffeomorphism $w_i:\bar{J}_i\rightarrow \bar{J}_{i+1},1\leq i \leq t+m-1$ if $c_i>c$, whose derivative of the endpoints is equal to $-\beta$. Define a map $g_\beta$ by
$$
g_\beta=\left \{
\begin{array}{ll}
h_{s_n s_{n-1} \cdots s_0}(x),& x\in L(\bar{J}_{s_n s_{n-1} \cdots s_0})\\
r_{s_n s_{n-1} \cdots s_0}(x),& x\in R(\bar{J}_{s_n s_{n-1} \cdots s_0})\\
f(x),& x\in \bar{J}_0\\
g_i(x),& x\in \bar{J}_i,c_i<c \\
w_i(x),& x\in \bar{J}_i,c_i>c,
\end{array}
\right.
$$
then we extend the map $g_\beta$ to the interval $I_\beta$, and the map on the interval $I_\beta$ is still denoted by $g_\beta$.

{\bf Step 2:} $g_\beta$ is a $C^1$ map.

We show $g_\beta$ is differentiable.

Let $c'$ be the critical point of $g_\beta$. $g_\beta$ is differentiable at $y\in I_\beta\backslash A_\beta$ by the definition of $g_\beta$. For $y\in A_\beta$, without loss of generality, we only consider the case $y<c'$. Notice that $g_\beta$ is monotonically increasing and decreasing on the left and the right of the critical point $c'$ respectively. To show that $g_\beta$ is differentiable, we only need to show that
$$\lim\limits_{z\downarrow y}\frac{m(g_\beta([y,z]))}{m([y,z])}$$
exists, where $m(A)$ denotes the Lebesgue measure of $A$. The proof of the left derivative is similar.

Next, we show that
\begin{equation}\label{left derivative}
\ \ \ \ \ \ \ \ \ \ \ \ \ \ \ \ \ \ \ \ \ \ \ \ \ \ \lim\limits_{z\downarrow y}\frac{m(g_\beta([y,z]))}{m([y,z])}=\beta.
\end{equation}
Let $X_1=[y,z]\cap B_\beta$ and $X_2=[y,z]\cap A_\beta$. Notice that $m(g_\beta(X_2))=\beta m(X_2)$. Therefore, to show (\ref{left derivative}), we only need to show
\begin{equation}\label{derivative}
\ \ \ \ \ \ \ \ \ \ \ \ \ \ \ \ \ \ \ \ \ \ \ \ \ \ \lim\limits_{z\downarrow y}\frac{m(g_\beta(X_1))}{m(X_1)}=\beta.
\end{equation}

We shall adapt the idea in  \cite{de2012one}(Page 42, Theorem 2.3) to prove (\ref{derivative}). If $y$ is the left endpoint of an inserted interval, then (\ref{derivative}) holds by the definition of $g_\beta$. If $y$ is not the endpoint of any inserted interval, then one can check that the right neighbor $[y,z]$ of $y$ intersects with infinite number of inserted intervals $J_{s_n s_{n-1} \cdots s_0}$. In this case $m(X_1)$ could be estimated from the sum of the measure of the inserted intervals completely contained in $[y,z]$ plus a piece of the inserted interval that contains $z$. Let
$$J(z)=\{J_{s_n s_{n-1} \cdots s_0}|J_{s_n s_{n-1} \cdots s_0}\cap[y,z]\neq \emptyset,n\in N\},$$
$$J'(z)=\{J_{s_n s_{n-1} \cdots s_0}|J_{s_n s_{n-1} \cdots s_0}\subset [y,z],n\in N\},$$
then
$$\sum\limits_{J_{s_n s_{n-1} \cdots s_0}\in J'(z)}a_{\beta,n}\leq m(X_1)\leq \sum\limits_{J_{s_n s_{n-1} \cdots s_0}\in J(z)}a_{\beta,n}.$$
Since $J_{s_n s_{n-1} \cdots s_0}$ is contained in $[y,z]$ if and only if $J_{s_{n-1} s_{n-2} \cdots s_0}$ is contained in $g_\beta([y,z])$, thus
$$\sum\limits_{J_{s_n s_{n-1} \cdots s_0}\in J'(z)}a_{\beta,n-1}\leq m(g_\beta(X_1))\leq \sum\limits_{J_{s_n s_{n-1} \cdots s_0}\in J(z)}a_{\beta,n-1}.$$
Therefore
\begin{equation}\label{diffeomorphism of one}
\frac{\sum\limits_{J_{s_n s_{n-1} \cdots s_0}\in J'(z)}a_{\beta,n-1}}{\sum\limits_{J_{s_n s_{n-1} \cdots s_0}\in J(z)}a_{\beta,n}}\leq \frac{m(g_\beta(X_1))}{m(X_1)}\leq\frac{\sum\limits_{J_{s_n s_{n-1} \cdots s_0}\in J(z)}a_{\beta,n-1}}{\sum\limits_{J_{s_n s_{n-1} \cdots s_0}\in J'(z)}a_{\beta,n}}.
\end{equation}
By the construction, $J'(z)\subset J(z)$, $\sharp\{J(z)\backslash J'(z)\}\leq 1$, and
$$\inf\{n|J_{s_n s_{n-1} \cdots s_0}\in J(z)\}\rightarrow +\infty,\quad \inf\{n|J_{s_n s_{n-1} \cdots s_0}\in J'(z)\}\rightarrow +\infty \quad  \quad (z\rightarrow y).$$
Since
$$\lim\limits_{n\rightarrow +\infty}a_{\beta,n}=\lim\limits_{n\rightarrow +\infty}\frac{1}{\beta^n}\frac{2}{n(n+1)}=0,$$
$$\lim\limits_{n\rightarrow +\infty}a_{\beta,n-1}/a_{\beta,n}=\lim\limits_{n\rightarrow +\infty}\frac{1}{\lambda_{\beta,n}}=\lim\limits_{n\rightarrow +\infty}\frac{\beta n+\beta}{n-1}=\beta,$$
then the outer terms in (\ref{diffeomorphism of one}) tend to $\beta$ as $z\downarrow y$. In fact,
\[\lim\limits_{z\downarrow y}\frac{\sum\limits_{J_{s_n s_{n-1} \cdots s_0}\in J'(z)}a_{\beta,n-1}}{\sum\limits_{J_{s_n s_{n-1} \cdots s_0}\in J(z)}a_{\beta,n}}=\lim\limits_{z\downarrow y}\frac{\sum\limits_{J_{s_n s_{n-1} \cdots s_0}\in J'(z)}a_{\beta,n-1}}{\sum\limits_{J_{s_n s_{n-1} \cdots s_0}\in J'(z)}a_{\beta,n}}=\lim\limits_{n_1\rightarrow +\infty \atop n_1<n_2<\cdots}\frac{\sum\limits_{i=1}^\infty a_{\beta,n_i-1}}{\sum\limits_{i=1}^\infty a_{\beta,n_i}}\]
\[=\lim\limits_{n_1\rightarrow +\infty \atop n_1<n_2<\cdots}\frac{\sum\limits_{i=1}^\infty \frac{1}{\beta^{n_i-1}}\frac{2}{n_i(n_i-1)}}{\sum\limits_{i=1}^\infty \frac{1}{\beta^{n_i}}\frac{2}{n_i(n_i+1)}}=\lim\limits_{n_1\rightarrow +\infty \atop n_1<n_2<\cdots}\beta\frac{\sum\limits_{i=1}^\infty \frac{1}{\beta^{n_i}}\frac{1}{n_i^2}}{\sum\limits_{i=1}^\infty \frac{1}{\beta^{n_i}}\frac{1}{n_i^2}}=\beta.\]
Similarly, the limit of the right side in (\ref{diffeomorphism of one}) is equal to $\beta$.

We show $g'_\beta$ is continuous.

By the definition of $g_\beta$, $g'_\beta$ is continuous on $B_\beta$. To prove that $g'_\beta$ is continuous on the interval $I_\beta$,  we only need to prove $g'_\beta$ is continuous on the Cantor set $A_\beta$. For $y\in A_\beta$, without loss of generality, we only consider continuity of left derivative for $y\in L(A_\beta)$. Notice that $g'_\beta(y)=\beta$ as $y\in L(A_\beta)$. Then we only need to show that
\begin{equation}\label{conti of deri}
\ \ \ \ \ \ \ \ \ \ \ \ \ \ \ \ \ \ \ \ \ \ \ \ \ \ \lim\limits_{z\downarrow y}g'_\beta(z)=g'_\beta(y)=\beta.
\end{equation}
According to the proof of differentiability of $g_\beta$, one can check that
$$\lim\limits_{z\downarrow y \atop z\in A_\beta}g'_\beta(z)=\beta.$$
To prove (\ref{conti of deri}), it is necessary to prove that
\begin{equation}\label{continu}
\ \ \ \ \ \ \ \ \ \ \ \ \ \ \ \ \ \ \ \ \ \ \ \ \ \ \lim\limits_{z\downarrow y \atop z\in B_\beta}g'_\beta(z)=\beta.
\end{equation}
If $y$ is the left endpoint of a inserted interval, then (\ref{continu}) holds by the definition of $g_\beta$. If $y$ is not the endpoint of any inserted interval, the same as before, we have
$$n_1:=\inf\{n|J_{s_n s_{n-1} \cdots s_0}\in J(z)\}\rightarrow +\infty \quad  \quad (z\rightarrow y).$$
Put $J_{s_{n_1} s_{n_1-1} \cdots s_0}=[u_1,v_1],J_{s_{n_1-1} s_{n_2-2} \cdots s_0}=[u_2,v_2]$. By the definition of $h_{s_{n_1} s_{n_1-1} \cdots s_0}$, we have
$$h'_{s_{n_1} s_{n_1-1} \cdots s_0}(x)=\beta+\frac{6[(v_2-u_2)-\beta(v_1-u_1)]}{(v_1-u_1)^3}(v_1-x)(x-u_1).$$
Notice that
$$v_1-u_1=\frac{n_1-1}{\beta(n_1+1)}(v_2-u_2).$$
Then
$$\beta \leq h'_{s_{n_1} s_{n_1-1} \cdots s_0}\leq \beta+\frac{6[(v_2-u_2)-\beta(v_1-u_1)]}{(v_1-u_1)^3}(v_1-\frac{u_1+v_1}{2})(\frac{u_1+v_1}{2}-u_1)=\beta+\frac{6}{n_1-1}.$$
Therefore,
$$\lim\limits_{n_1\rightarrow +\infty}h'_{s_{n_1} s_{n_1-1} \cdots s_0}(x)=\beta, \quad x\in J_{s_{n_1} s_{n_1-1} \cdots s_0}.$$
It follows that (\ref{continu}) holds.

{\bf Step 3:} $g_\beta$ admits a thick hyperbolic repelling invariant Cantor set which contains a wild attractor.

The Lebesgue measure of $A_\beta$ is equal to 1 since $C_\beta$ is a countable set on the unit interval $I$. One can check that $A_\beta$ is a Cantor set since $A_\beta$ is the complement of an open set $B_\beta$ and $C_\beta$ is dense on the unit interval $I$ by  \cite{cui2010alpha}. Notice that $|g'_\beta(x)|=\beta>1$ for each point $x\in A_\beta$ and $g_\beta(A_\beta)\subset A_\beta$ following from $g_\beta^{-1}(B_\beta)\subset B_\beta$, then $A_\beta$ is also a hyperbolic repelling invariant set. Furthermore, $A_\beta$ contains a wild attractor. We divide the proof into two cases: $\beta\in(\sqrt{2}, 2]\cap D$ and $\beta\in(1,\sqrt{2}]\cap D$.

{\bf Case A:} $\beta\in(\sqrt{2}, 2]\cap D$.

In the case $\beta\in(\sqrt{2}, 2]\cap D$, $f_\beta$ is not renormalizable, which implies $f_\beta$ is topologically transitive in the core $P_\beta$ of the dynamical system $(I,f_\beta)$, therefore one can check that $P_\beta$ is both topological and metric attractor on $(I,f_\beta)$. After our surgery, this attractor was broken down. Notice that $B_\beta$ except for countable set $\Lambda_\beta$ eventually enters the periodic intervals inserted at the orbit $orb(c_t)$ through countable iterations of $g_\beta$. In fact, one can check that $\Lambda_\beta=\cup_{i=0}^{+\infty}g_\beta^{-i}(c')$, where $c'$ is the critical point of $g_\beta$.
Let $\tilde{P}_\beta$ be the core on the system $(I_\beta,g_\beta)$, corresponding to the core $P_\beta$ of the system $(I,f_\beta)$. Notice that $g_\beta(\tilde{P}_\beta)=\tilde{P}_\beta$. Once $A_\beta$ enters $\tilde{P}_\beta$, $A_\beta$ will be restricted to $\tilde{P}_\beta$. Define $A_\beta^\ast:=A_\beta\cap \tilde{P}_\beta$. It is obviously seen that $A_\beta^\ast$ is a forward invariant closed set. Furthermore, one cancheck that there exists and only $A_\beta\cup\Lambda_\beta$ is attracted by $A_\beta^\ast$, i.e., $\mathbb{B}(A_\beta^\ast)=A_\beta\cup\Lambda_\beta$.

Next, we will show $A_\beta^\ast$ is a wild attractor of $g_\beta$. Since $A_\beta\cup\Lambda_\beta$ is nowhere dense in $I_\beta$, then there exists no open interval in $I_\beta$ attracted by $A_\beta^\ast$, which follows that $A_\beta^\ast$ is not a topological attractor. The basin of attraction $\mathbb{B}(A_\beta^\ast)$ of $A_\beta^\ast$ has Lebesgue measure equal to 1.
$g_\beta|A_\beta^\ast$ is topologically transitive since $f_\beta|P_\beta$ is topologically transitive, which follows that $A_\beta^\ast$ is not decomposable by  \cite{akin2012conceptions}, then there exists no forward invariant closed proper subset $A_0$ of $A_\beta^\ast$ such that $\mathbb{B}(A_0)$ admits positive Lebesgue measure. Therefore, $A_\beta^\ast$ is a metric attractor.

{\bf Case B:} $\beta\in(1,\sqrt{2}]\cap D$.

In the case $\beta\in(1,\sqrt{2}]\cap D$, $f_\beta$ can be renormalized finite times via period-doubling renormalization  \cite{brucks2004topics}. More precisely, if $\sqrt{2}< \beta^m \leq 2$ for some $m:=2^k, k\geq 1$, then $f_\beta$ is $k$ times renormalizable. Here we only consider the case $k=1$ since the case $k\geq 2$ is similar. If $k=1$, then there exists a restrictive interval $\tilde{R}_\beta$ such that $g_\beta^2\mid \tilde{R}_\beta$ is again a unimodal map. Denote the core of $g_\beta^2\mid \tilde{R}_\beta$ by $\tilde{P}_\beta^1$ . Similar to the proof of the case A, $A_\beta^\ast=A_\beta \cap \tilde{P}_\beta^1 \cup g_\beta (\tilde{P}_\beta^1)$ is a wild attractor on the dynamical system $(I_\beta,g_\beta)$. Similarly, if $k\geq 2$, then $A_\beta\cap (\cap_{i=0}^{2k-1} g_\beta^i(\tilde{P}_\beta^k))$ is a wild attractor of $g_\beta$, where $\tilde{P}_\beta^k$ is the core of $g_\beta^{2^k}\mid \tilde{R}_\beta^k$ with restrictive interval $\tilde{R}_\beta^k$.

{\bf Step 4:} Affine diffeomorphism.

Suppose $I_\beta=[a_\beta,b_\beta]$. Define an affine map $\varphi(x):[a_\beta,b_\beta]\rightarrow [0,1],x\mapsto \frac{x-a_\beta}{b_\beta-a_\beta}$ and  $\tilde{g}_\beta=\varphi\circ g_\beta \circ \varphi^{-1}$. Then $\tilde{g}_\beta$ is a $C^1$ map since $\varphi$ and $\varphi^{-1}$ are $C^1$. One can check that $|\tilde{g}'_\beta(x)|=\beta>1$ for each $x\in \varphi(A_\beta)$ since $\varphi$ is a affine diffeomorphism, then $\tilde{g}_\beta$ is a unit interval unimodal map admitting a hyperbolic repelling invariant set $\varphi(A_\beta)$ which contains a wild attractor since $\tilde{g}_\beta$ and $g_\beta$ are topologically conjugated. Notice that the length of inserted intervals differs by a finite number from $\sum\limits_{n=1}^{\infty}F_{\beta,n} a_{\beta,n}.$
In fact, we just changed a finite number in $\{a_{\beta,n}\}$ because we controled the length of the intervals inserted at the points of $orb(c)$ to be 1. By the Proposition \ref{summable}, the total length of inserted intervals is finite for each $\beta\in D$, which follows that the hyperbolic repelling invariant set of $\tilde{g}_\beta$ admits positive Lebesgue measure. Denote the hyperbolic repelling invariant Cantor set of $\tilde{g}_\beta$ by $\tilde{A}_\beta$ and the wild attractor of $\tilde{g}_\beta$ by $\tilde{A}_\beta^\ast$. One can check that $\mathbb{B}(\tilde{A}_\beta^\ast)=\tilde{A}_\beta\cup \tilde{\Lambda}_\beta$, where $\tilde{\Lambda}_\beta$ corresponds to the countable set $\Lambda_\beta$, i.e., $\tilde{\Lambda}_\beta=\varphi(\Lambda_\beta)$.

The proof of Theorem A is complete.
$\hfill\square$
\end{proof}

\section*{Acknowledgements}
\noindent

This work was supported by the National Key Research and Development Program of China (No. 2020YFA0714200).

\section*{References}
\setcounter{equation}{0}




\end{document}